\documentclass[12pt]{amsart}

\usepackage{amssymb,amsmath,amsthm,comment}
\usepackage{amsfonts}
\usepackage{pdfpages,graphicx} 
\usepackage[section]{placeins} 
\usepackage{wrapfig}
\usepackage{float}
\usepackage{lineno}
\usepackage{youngtab}

\usepackage{setspace}
\newcommand{\shrinkmargins}[1]{
  \addtolength{\textheight}{#1\topmargin}
  \addtolength{\textheight}{#1\topmargin}
  \addtolength{\textwidth}{#1\oddsidemargin}
  \addtolength{\textwidth}{#1\evensidemargin}
  \addtolength{\topmargin}{-#1\topmargin}
  \addtolength{\oddsidemargin}{-#1\oddsidemargin}
  \addtolength{\evensidemargin}{-#1\evensidemargin}
  }
\shrinkmargins{0.5}

\newtheorem{theorem}{Theorem}
\newtheorem{lemma}[theorem]{Lemma}

\newtheorem{corollary}[theorem]{Corollary}
\newtheorem*{theorem*}{Theorem}

{Claim}

\theoremstyle{definition}
\newtheorem{definition}[theorem]{Definition}

\theoremstyle{remark}

\newtheorem*{remark}{Remark}

\numberwithin{theorem}{section} \numberwithin{equation}{section}

\usepackage{caption}
\usepackage{multirow}
\makeatletter
\@namedef{subjclassname@2020}{\textup{2020} Mathematics Subject Classification}
\makeatother

\def\func#1{\mathop{\rm #1}}%

\newcommand{\floor}[1]{\left\lfloor #1 \right\rfloor}

\begin{document}
\title[Transfer]{Zeros Transfer For Recursively defined Polynomials}
\author{Bernhard Heim }
\address{Faculty of Mathematical and Natural Sciences, Mathematical Institute, University of Cologne, Weyertal 86--90, 50931 Cologne, Germany}
\email{bheim@uni-koeln.de}
\address{Lehrstuhl A f\"{u}r Mathematik, RWTH Aachen University, 52056 Aachen, Germany}
\email{bernhard.heim@rwth-aachen.de}
\author{Markus Neuhauser}
\address{Kutaisi International University, 5/7, Youth Avenue,  Kutaisi, 4600 Georgia}
\address{Lehrstuhl A f\"{u}r Mathematik, RWTH Aachen University, 52056 Aachen, Germany}
\email{markus.neuhauser@kiu.edu.ge}
\author{Robert Tr\"{o}ger}
\email{robert@silva-troeger.de}
\begin{abstract}
The zeros of D'Arcais polynomials, also known as Nekrasov--Okounkov polynomials, 
dictate the vanishing of the Fourier coefficients of powers of the Dedekind functions.
These polynomials satisfy difference equations of hereditary type with non-constant coefficients.
We relate the D'Arcais polynomials to polynomials satisying a Volterra difference equation of convolution type. We obtain results on the transfer of the location of the zeros.

As an application, we obtain an identity between Chebyshev
polynomials of the second kind and $1$-associated Laguerre polynomials.
We obtain a new version of the Lehmer conjecture and bounds for the zeros of the Hermite polynomials.
\end{abstract}
\maketitle
\newpage
\section{Introduction}
Powers of the Dedekind $\eta$-function $\eta \left(
\omega \right) ^{r}$ and properties of the 
Fourier coefficients, especially their non-vanishing, play an important role in
the theory of numbers \cite{Se85,On03,BCO22}, combinatorics \cite{An99,Ha10}, and physics \cite{NO06,DDMP05,DD22}. 
Let $q:=\textrm{e}^{2 \pi \textrm{i} \omega}$, where $\omega$ is in the complex upper half-plane. Let $r \in \mathbb{N}$:
\begin{equation*}
\sum_{n=0}^{\infty} a_n(r) \, q^n =q^{-\frac{r}{24}}\,\,
\eta \left( \omega 
\right) ^{r}=  \prod_{n=1}^{\infty} \left( 1 - q^n\right)^{r}.
\end{equation*}
Let $r$ be even. Then Serre \cite{Se85} proved that $\{a_n(r)\}_n$ is lacunary if and only if 
$$r \in S:= \{2,4,6,8,10,14,26\}.$$
It is conjectured by Lehmer \cite{Le47, On08,BCO22}
that $a_n(24)$
never vanishes.
Numerical experiments \cite{HNW18} suggest that $a_n(r) \neq 0$ for all $n,r \in \mathbb{N}$
with $r$ even and $r
\not\in S$, generalizing the Lehmer conjecture.

It is known that the coefficients $a_n(r)$, are special values of polynomials 
$P_n(z)$ at $z=-r$ of degree $n$, 
the D'Arcais polynomials \cite{DA13, Ne55, HN20a},
also known as 
the Nekrasov--Okounkov polynomials \cite{Ha10, Zh22}
in combinatorics. We have
\begin{equation*}
\sum_{n=0}^{\infty} P_n(z) \, q^n = \prod_{n=1}^{\infty} \left( 1 - q^n\right)^{-z} = \exp \left( z\sum_{n=1}^{\infty} \sigma_1(n) \,\frac{q^n}{n}\right), \quad
z \in \mathbb{C}
,
\end{equation*}
where $\sigma_d(n):= \sum_{\ell \mid n} \ell^d$. 
Also,
$P_n(z)$ is integer-valued and
$n! \, P_n(z)$ is monic of degree $n$ with non-negative integer coefficients and therefore,
zeros are algebraic integers. 
The polynomials can also be defined
recursively \cite{Ne55, HNT20}, 
which enables us
to study them
using methods from difference equations:
\begin{equation}\label{def:general}
P_n(z)= \frac{z}{n} \sum_{k=1}^n \sigma_1(k) \,P_{n-k}(z), \quad
n \geq 1
,
\end{equation}
with initial value $P_0(z)=1$. This generalizes the well known
recurrence relation for
partition numbers $p(n)$. Since the times of Euler, it is known that
\begin{equation*}\label{Euler}
n \, p(n) = \sum_{k=1}^n \sigma_1(k) \, p(n-k).
\end{equation*}

Serre's results \cite{Se85} imply that for each $z \in S$,
there are infinitely many $n$, such that $P_n(-z)=0$. 
It would be interesting to devise
a combinatorial proof utilizing (\ref{def:general}).
Consider $z$ as a parameter and (\ref{def:general}) as a difference equation. We refer to Elaydi's excellent introduction to difference equations
\cite{El05}. The equation (\ref{def:general}) has non-constant coefficients and is of hereditary type.

Poincar\'{e} and Perron (\cite{El05},
section 8.2) offered a method to solve difference equations of fixed order with non-constant coefficients,
to obtain the asymptotic behavior of the solutions. On the other hand, for some difference equations of hereditary type with constant coefficients,
called Volterra difference equations of convolution type, one can utilize the discrete Laplace transform, also
called Z-transform (\cite{El05}, section 6.3). 

In this paper,
we develop a new method to study the solutions, and especially the zero
distribution of (\ref{def:general}). We
associate
$\{P_n(z)\}_n$ with another family of polynomials 
$\{Q_n(z)\}$. These polynomials satisfy a Volterra difference equation. Then we transfer properties of $Q_n(z)$
to the D'Arcais polynomials $P_n(z)$. In this paper,
we
provide evidence that $Q_n(z)$ is easier to study,
rather than $P_n(z)$.
\begin{definition} \label{def}
Let $g$ be a normalized arithmetic function with non-negative values and
let $h$ be a normalized arithmetic function with positive values.
Then,
\begin{equation}\label{recursion}
P_n^{g,h}(z) := \frac{z}{h(n)} \, \sum_{k=1}^n g(k) \, P_{n-k}^{g,h}(z), \qquad
n \geq 1
,
\end{equation}
with initial value $P_0^{g,h}(z)=1$.
\end{definition}
Examples of arithmetic functions are provided by $g\left( n\right) =\sigma_d(n), \psi_d(n)=n^d$, 
$h_s(n)=n^s$ for $s \in [0,1]$ and $\func{id}(n)=n$. Note,
that $P_n(z)= P_n^{\sigma_1, h_{1}}\left( z\right) $. To simplify notation and to highlight the special roles of $s=0$ and $s=1$,
we put $P_n^g(z):= P_n^{g,h_1}(z)$ and 
$Q_{n}^{g}(z):= P_n^{g,h_0}(z)$. 

Although, we
address the general case in this paper, we begin to illustrate our results
with the following explicit example related to orthogonal polynomials \cite{Sz75,Ch11,Do16} related to
$g(n)= \func{id}(n)$.

\subsection{Zero Transfer from Chebyshev to Laguerre
Polynomials}
We describe how properties of Chebyshev polynomials, provided by
$h(n)=h_0(n)=1$, give obstructions for the zero distribution of 
associated Laguerre polynomials $L_m^{(1)}(z)$. Orthogonal 
polynomials have real zeros, which are interlacing and well-studied. 
To quote Rahmann--Schmeisser (\cite{RS02}, introduction, page 24):
``The Chebyshev polynomials are the only classical orthogonal polynomials 
whose zeros can be determined in explicit form".
This makes our task
even more appealing, since we
claim that these explicit values can be utilized to study the zeros of Laguerre polynomials.

Let $L_m^{(\alpha)}(z)$ denote the $m$th $
\alpha
$-associated Laguerre polynomial and $U_m(z)$ represent the
$m$th Chebyshev polynomial of second kind. Then we have (\cite{HLN19}, lemma 3.3 and \cite{HNT20}, remark 2.8):
\begin{eqnarray*}
P_n^{\func{id}}(z) & = & \frac{z}{n} \, L_{n-1}^{(1)} (-z),\\
Q_n^{\func{id}}(z) & = & z \, U_{n-1}\left(\frac{z}{2}+1\right).
\end{eqnarray*}
We have the following result. 
\begin{theorem}
Let $n \geq 2$. Let $\alpha_n$ and $\beta_n$ be the smallest and
the largest zeros of 
$Q_n^{\func{id}}(z)/ \, z= U_{n-1}\left(z/2+1\right)$, where
\[
\alpha_n = 2 \, \cos \left( \frac{n-1}{n} \, \pi \right) -2,
\qquad
\beta_n= 2 \, \cos \left( \frac{1}{n} \, \pi \right) -2 .
\]
Then the zeros
of $ n \, P_n^{\func{id}}(z)/\, z= L_{n-1}^{(1)}(-z)$ are contained in
the interval
$$\left[ \alpha_n \, (n-1); \beta_n \right].$$
\end{theorem}
\begin{corollary}
Let $
m \geq 2$. Then the
zeros of $L_m^{(1)}(z)$ are contained in the interval
\begin{equation*}
\left[
2-2 \, \cos \left( \frac{\pi}{m+1} \right)
\, ;
\left( 2-2 \, \cos \left( \frac{m \, \pi}{m+1} \right)
\right) \, m\right].
\end{equation*}
\end{corollary}
Let $Q_n^g(x)$ be given. Then we denote by $\alpha_n$ and $\beta_n$ the smallest and the largest real zeros. 
Let $P_n^g(x)$ be given. 
Then we denote by $\tilde{\alpha }_{n}$ and
$\tilde{\beta }_{n}$ the smallest and the largest real zeros. 
\begin{table}[H]
\[
\begin{array}{rrrr|r}
\hline
n&\alpha_{n}& \beta_{n}&\frac{\tilde{\alpha}_{n}}{(n-1) \alpha_{n}}&\frac{\tilde{\beta}_{n}}
{\left( n-1\right) \beta_{n}}\\ \hline \hline
2 & -2.0000 & -2.0000 & 1.0000 & 1.0000 \\
3 & -3.0000 & -1.0000 & 0.7887 & 0.6340 \\
4 & -3.4142 & -0.5858 & 0.7575 & 0.5325 \\
5 & -3.6180 & -0.3820 & 0.7569 & 0.4865 \\
6 & -3.7321 & -0.2679 & 0.7642 & 0.4606 \\
7 & -3.8019 & -0.1981 & 0.7736 & 0.4440 \\
8 & -3.8478 & -0.1522 & 0.7831 & 0.4326 \\
9 & -3.8794 & -0.1206 & 0.7922 & 0.4243 \\
10 & -3.9021 & -0.0979 & 0.8007 & 0.4179 \\
20 & -3.9754 & -0.0246 & 0.8559 & 0.3926 \\
100 & -3.9990 & -0.0010 & 0.9422 & 0.3757 \\
120 & -3.9993 & -0.0007 & 0.9483 & 0.3751 \\
200 & -3.9998 & -0.0002 & 0.9622 & 0.3738 \\ \hline
\end{array}
\]
\caption{\label{alphabeta}
Approximative values for the smallest $\alpha_n$  and the largest zeros $\beta_n$ of
$Q_n^{\func{id}}(z)/z$, compared to the smallest 
$\tilde{\alpha}_n$ and the largest zeros $\tilde{\beta}_n$ of
$P_n^{
\func{id}}(z)/z$.}
\end{table}

Theorem \ref{links} implies that $\frac{\tilde{\alpha }_{n}}{(n-1) \, \alpha _{n}} <1$
for $g(n)=n$ as
illustrated by Table \ref{alphabeta}. We expect that
$$ \lim_{n \rightarrow \infty} \frac{\tilde{\alpha }_{n}}{(n-1) \, \alpha _{n}} =1.$$
Szeg\H{o} (\cite{Sz75}, (6.32.6)) offers for  $\tilde{\alpha }_{n}$ the
approximation $\gamma_n$ given by
\begin{equation*}
\gamma_n = 
\left[ \sqrt{4n+4}-6^{-1/2}\left( 4n+4\right) ^{-1/6}i_{1}\right] ^{2},
\end{equation*}
where $i_1$ denotes the smallest positive zero of Airy's function $A(x)$
(we also refer
to Table \ref{Sz}). 
\begin{table}
\[
\begin{array}{rrrr}
\hline
n&\tilde{\alpha }_{n}&\left( n-1\right) \alpha _{n}&
\gamma_n 
\\ \hline \hline
2&-2.000000&-2.000000&-5.007008\\
3&-4.732051&-6.000000&-8.014282\\
4&-7.758770&-10.242641&-11.194149\\
5&-10.953894&-14.472136&-14.488083\\
6&-14.260103&-18.660254&-17.863804\\
7&-17.645964&-22.811626&-21.301454\\
8&-21.092177&-26.934313&-24.787866\\
9&-24.585955&-31.035082&-28.313824\\
10&-28.118343&-35.119017&-31.872599\\
20&-64.649712&-75.532157&-68.531759\\
30&-102.253573&-115.682270&-106.182727\\
40&-140.359594&-155.759552&-144.313504\\
50&-178.767074&-195.806619&-182.736279\\
60&-217.379108&-235.838285&-221.358695\\
70&-256.140634&-275.861043&-260.127718\\
80&-295.017025&-315.878188&-299.009766\\
90&-333.984925&-355.891567&-337.982076\\
100&-373.027751&-395.902299&-377.028430\\ \hline
\end{array}
\]
\caption{\label{Sz}Approximative values for the smallest zeros $\tilde{\alpha }_{n}$ of $P_{n}^{\func{id}}\left( z\right) /z$, compared with (\cite{Sz75}, (6.32.6)).}
\end{table}

Further, Theorem \ref{rechts}
implies that
$\frac{\beta_n}{\tilde{\beta }_{n}}<1$. Nevertheless, Table \ref{alphabeta}
indicates that 
$\frac{\tilde{\beta}_n}{(n-1) \, \beta_n}$ converges against $\approx 0.37\ldots$.
If this is the case, it would be interesting to identify this constant in the context of orthogonal polynomials. Thus, we
raise the question, if the following limit exists:
 \begin{equation*}
 \lim_{n \rightarrow \infty}
 \frac{\tilde{\beta}_n}{(n-1) \, \beta_n} = \beta .
\end{equation*}

\subsection{D'Arcais Polynomials}
Real zeros are related to the sign changes of the sequences 
$\{ P_n^{g,h}(z_0)\}_n $ for fixed $z_0 \in \mathbb{R}_{<0}$.

Note, that the non-trivial real zeros are negative, since the coefficients of 
$P_n^{g,h}(z)$ are non-negative real numbers. The real zeros dictate the sign changes of
$\{P_n^{g,h}(z_0)\}_n$ for fixed $z_0 \in \mathbb{R}_{<0}$. We also refer to \cite{HN20b}, where the sign changes of the Ramanujan $\tau$-function
are analyzed.

Kostant \cite{Ko04}, building on representation theory of complex Lie algebras 
and Han \cite{Ha10} on the Nekrasov--Okounkov hook length formula, 
proved that $(-1)^n \, P_n^{\sigma}(z)> 0$ for $ z \leq -n^2 +1 $ for $n \geq 4$. 
Utilizing the recursion formula (\ref{recursion}) of $P_n^{g,h}(z)$
and assuming that $\sum_{n=1}^{\infty} g(n) \, q^n$ is regular at $q=0$ and $h$ monotonously increasing, 
there exists a $\kappa >0$, only dependent
on $g$, such that
$P_n^{g,h}(z) \neq 0$ for all complex $z$, such that $\vert z \vert > \kappa\, h(n-1)$.
For example, we have shown \cite{HN21} that the D'Arcais polynomials $P_n(z)$ are non-vanishing for all complex $z$ with $\vert z \vert > 10.\overline{81}
\, (n-1)$.

In \cite{HNT20}, we obtained the following numerical result.
Let $\alpha_n$ be the smallest real zero of $Q_n^{\sigma}(z)$. Then the smallest real zero
$\tilde{\alpha}_n$ of $P_n^{\sigma}(z)$ satisfies $ \alpha_n < \frac{\tilde{\alpha}_n}{n-1}$ 
for $n \leq 1400$
(we refer to Figure \ref{fig2}, \cite{HNT20}).

Let $\kappa_n$ 
be the maximum of $\{ \vert \alpha_1\vert, \ldots, \vert \alpha_n \vert \}$ 
and $\tilde{\kappa}_n$ be the maximum of $\{ \vert \tilde{\alpha}_1\vert, \ldots, \vert \tilde{\alpha}_n \vert \}$, where $\tilde{\alpha}_k$ is the smallest zero of
$P_k^{\sigma}(x)$. Then we prove in this paper applying Theorem \ref{links} that
\begin{equation*}
\frac{\tilde{\kappa }_{n}}{\kappa_n \,(n-1)} < 1.
\end{equation*}
Let $n \leq 1400$, then we have checked that $\alpha_n > \alpha_{n+1}$ and
$\tilde{\alpha}_n > \tilde{\alpha}_{n+1}$. This implies the results displayed in Figure \ref{fig2}, \cite{HNT20}.

\begin{figure}[ht]
  \centering
  \caption{Minimal real zeros of
  $Q_n(x)$ and $P_n^{\sigma} \left( (n-1)x \right)$.}
  \label{fig2}
  \includegraphics[width=0.6\textwidth]{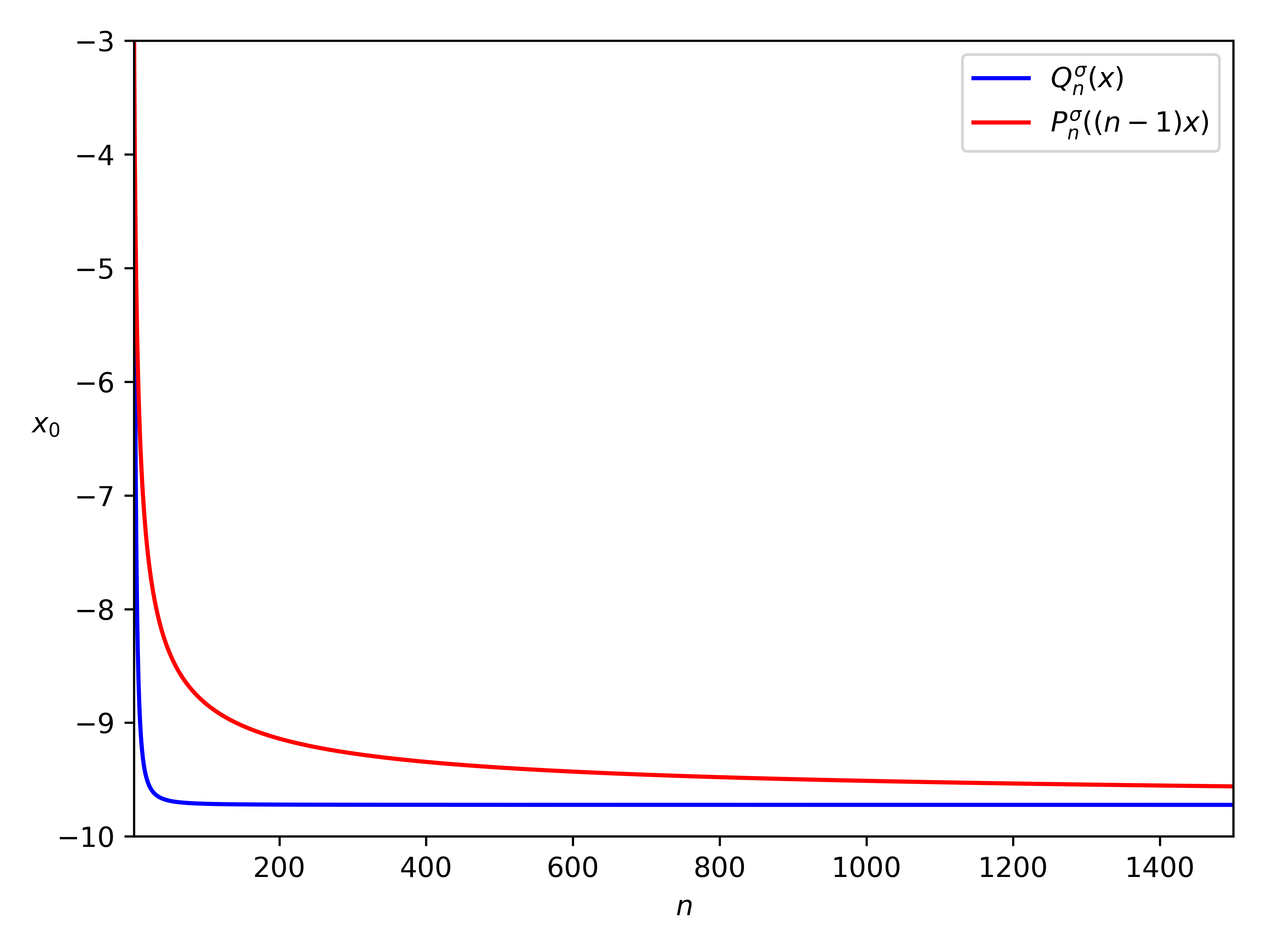}
\end{figure}
Table \ref{observation} illustrates that
$\left\vert \alpha _{n}\left( n-1\right) \right\vert $ is an upper bound for $\left\vert \tilde{\alpha }_{n}\right\vert $
and that the quotient seems to converge to $1$. Let $\beta_n$ be the largest
non-trivial real zero of $Q_n^{\sigma}(x)$ and
$\tilde{\beta }_{n}$ be the largest non-trivial zero of $P_n^{\sigma}(x)$, the D'Arcais polynomial of degree $n$. Let $n$ be a prime, then we observed for $n \leq 257$ that
\begin{equation*}
\frac{\tilde{\beta_n}}{(n-1) \beta_n} <1.
\end{equation*}
For general $n$, this is not always the case. For example, for $n=18$ we obtain
$\frac{\tilde{\beta }_{n}}{\left( n-1\right) \beta _{n}}\approx 1.878282$.
Moreover, it seems that the quotient for
a prime number $n$ converges against a positive real number, if $n$ goes to infinity. This indicates that the D'Arcais polynomials with prime number degree has some special properties among all D'Arcais polynomials.

\begin{table}
\[
\begin{array}{rccc|c}
\hline
n&\alpha _{n}&\beta _{n}&\frac{\tilde{\alpha }_{n}}{\left( n-1\right) \alpha _{n}}&\frac{\tilde{\beta }_{n}}{\left( n-1\right) \beta _{n}}\\ \hline
2&-3.000000&-3.000000&1.000000&1.000000\\
3&-5.236068&-0.763932&0.763932&0.654508\\
5&-7.418833&-0.194397&0.694579&0.499140\\
7&-8.352996&-0.087008&0.697784&0.444219\\
11&-9.087471&-0.031512&0.728503&0.410896\\
13&-9.251318&-0.021917&0.743585&0.392685\\
17&-9.434121&-0.012353&0.769357&0.383976\\
19&-9.488052&-0.009767&0.780202&0.381648\\
23&-9.558851&-0.006544&0.798626&0.378863\\
47&-9.681142&-0.001500&0.860549&0.363084\\
149&-9.718238&-0.000145&0.928608&0.352456\\
257&-9.721056&-0.000049&0.948860&0.350195\\
\hline
\end{array}
\]
\caption{\label{observation}
Approximative values for the smallest $\alpha _{n}$ and the largest real zeros 
$\beta _{n}$ of $Q_{n}^{\sigma }\left( z\right) /z$, compared to the 
smallest $\tilde{\alpha }_{n}$ and the largest real zeros 
$\tilde{\beta }_{n}$ of $P_{n}^{\sigma }\left( z\right) /z$.}
\end{table}

Figure \ref{QP} compares the locations of the largest real roots of $Q_n^{\sigma}(x)$ and $P_n^{\sigma}(x)$, where $n$ is a prime number. Real zeros have a blue color and non-real zeros (their real part) are plotted in red.

\begin{figure}[H]
\includegraphics[width=.7\textwidth]{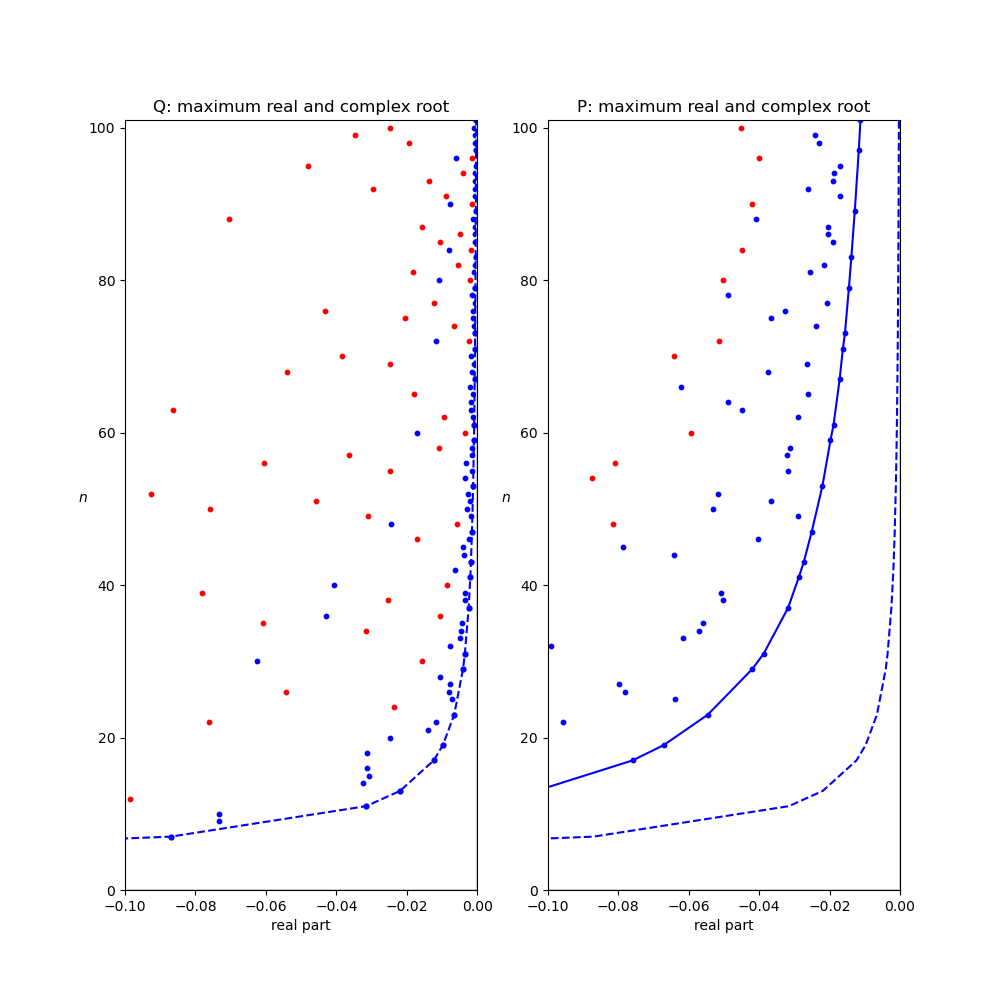}
\caption{\label{QP} Real part of the smallest non-trival zeros of $Q_{n}^{\sigma} (z) $  (left plot), 
and $P_n^{\sigma} (x)$ (right plot), where the dashed line denotes the convex
hull of the smallest real zero of $Q_n^{\sigma}(z)$
for $1\leq n\leq 100$, where $n$ is prime. }
\end{figure}

\subsection{Main Results}
We describe the impact of the smallest and the
largest real zeros of $\{Q_n^g(z)\}_n$ 
on the location of the real zeros of $\{P_n^{g,h}(z)\}_n$. Let $H(n):= \max \{0,h(k)\, : \, 1 \leq k \leq n\}$ and $H(0):=0$.
\begin{theorem}
\label{links}
Let $g$ and $h$ be normalized arithmetic functions
with $g,h: \mathbb{N} \longrightarrow \mathbb{R}_{\geq 0}$. 
Let us fix $n \geq 1$ and suppose there exists $\kappa_n > 0$, such that
$(-1)^m \, Q_m^g(x) >0$ for all real $x < - \kappa_n$ and $1 \leq m \leq n$. 
Let
\begin{equation}
y
\leq x \, H(n-1) \label{xH}
\end{equation}
then we have for $1 \leq m \leq n$ the inequalities
\begin{equation} \label{inequality}
(-1)^m P_m^{g,h}(y)  \geq  (-1)^m \frac{y}{x \, h\left(
m\right) } \, Q_m^{g}(x) >0.
\end{equation}
Let already  $y < \ - \kappa_n \, H(n-1)$ instead of (\ref{xH}), then we have for
$1 \leq m \leq n$: 
\begin{equation}\label{P}
(-1)^m P_m^{g,h}(y)   >0.
\end{equation}
\end{theorem}

Further, we obtain a result on the influence of the largest non-trivial zeros of 
$Q_m^g(z)$ for $1 \leq m \leq n$ on the zeros close to $0$ of $P_n^{g,h}(z)$.
\begin{theorem}
\label{rechts}
Let $g$ and $h$ be normalized positive valued arithmetic functions, 
$h(n) \geq 1$ for all $n \in \mathbb{N}$.
Let
$n \geq 1$ be fixed.
Suppose there exists a $\mu_n <0$, such that for all $ \mu_n < x <0$: $Q_m^g(x) < 0$ for $ 1 \leq m \leq n$. Then $P_m^{g,h}(x) < 0$ for all $1 \leq m \leq n$.
\end{theorem}

These results follow from an identity, which we consider as fundamental in studying
properties of $P_n^{g,h}\left( z\right) $.

\begin{lemma}[Fundamental transfer lemma]
\label{uebertragung} Let $g$ and $h$ be normalized positive valued arithmetic functions.
Let
$x,y\in \mathbb{C}\setminus \left\{ 0\right\} $. Then we have for $n \geq 1$:
\begin{equation}
\frac{h\left( n\right) }{y} \, P_{n}^{g,h}\left( y\right) -\frac{1}{x} \,Q_{n}^{g}\left( x\right) 
=
\sum _{k=1}^{n-1}\left( \frac{1}{x}-\frac{h\left( k\right) }{y}\right) P_{k}^{g,h}\left( y\right) Q_{n-k}^{g}\left( x\right)
.\label{ineq:uebertragung}
\end{equation}
\end{lemma}
The expression on the left hand
side of (\ref{ineq:uebertragung}) is equal to
\begin{equation*}
\sum_{k=1}^{n-1} g(k) \, \left( P_{n-k}^{g,h}(y) - Q_{n-k}^{g}\left( x\right) \right) .
\end{equation*}

To prove Theorem \ref{links} and Theorem \ref{rechts},
we examine the identity
\begin{equation}
\frac{ x h\left( n\right) }{y} \, P_{n}^{g,h}\left( y\right) -Q_{n}^{g}\left( x\right) 
=
\sum _{k=1}^{n-1}\left( 1 -\frac{  x \, h\left( k\right) }{y}\right) P_{k}^{g,h}\left( y\right) Q_{n-k}^{g}\left( x\right)
\label{Vorzeichen}
\end{equation}
for $\frac{x \, h(k)}{y} > 1$ or $\frac{x \, h(k)}{y} <1$ for $ 1 \leq k \leq n-1$.
\section{Proof of the Fundamental Transfer Lemma \ref{uebertragung}, Theorem \ref{links}, and Theorem \ref{rechts}}
We verify Theorem \ref{links} and Theorem \ref{rechts} for $n=1$ and $n=2$.
The fundamental transfer lemma is trivial for $n=1$ and reduces to the term 
$y-x$ on the left and right hand side for $n=2$,
since $h(2) P_2^{g,h}(x) = Q_2^{g,h}(x)= x ( x + g(2))$.

Let $n=1$. Let $\kappa_1, - \mu_1>0$ and $\mu_1 <0$ be any positive real numbers.
Then $y< x < - \kappa_1$ implies (\ref{inequality}). Let $y < 0$.
Then $-y>0$ (\ref{P}). The claim of Theorem~\ref{rechts} for $n=1$ is trivial.

Let $n=2$. Then any $\kappa_2 > g(2)$ works. Let $y<x < - \kappa_2$. Then we obtain
(\ref{inequality}) for $m=1$
and (\ref{inequality}) for $m=2$
since $y \geq x$. Further, let $y < - \kappa_2$. Then (\ref{P}) for $m=1$ is obvious
and (\ref{P}) follows for $m=2$, since $y < -\kappa_2 < -g(2)$.
It is sufficient and necessary to choose $\mu_2$
as $ -g(2) < \mu_2 <0$. 
Then Theorem \ref{rechts} follows.

\newpage

\subsection{Proof of the Fundamental Transfer Lemma}
\begin{proof}[Proof of Lemma~\ref{uebertragung}]
To make the proof transparent, we first assume that all involved sums are regular at $q=0$.
We have
\begin{eqnarray*}
x\sum _{k
=1}^{\infty }g\left(
k\right) q^{
k}
\sum _{m
=0}^{\infty }P_{m
}^{g,h}\left( y\right) q^{
m}
&=&\sum _{n=1}^{\infty }x\sum _{\substack{k\geq 1,m\geq 0 \\ k+m=n}}g\left( k\right) P_{m}^{g,h}\left( y\right) q^{k+m}\\
&=&\sum _{n=1}^{\infty }\frac{xh\left( n\right) }{y}\frac{y}{h\left( n\right) }\sum _{k=1}^{n}g\left( k\right) P_{n-k}^{g,h}\left( y\right) q^{n}\\
&=&
\sum _{n=1}^{\infty }\frac{xh\left( n\right) }{y}P_{n}^{g,h}\left( y\right)
q^{n}.
\end{eqnarray*}
Therefore,
\[
\left( 1-x\sum _{n=1}^{\infty }g\left( n\right) q^{n}\right) \sum _{n=0}^{\infty }P_{n}^{g,h}\left( y\right) q^{n}=\sum _{n=0}^{\infty }\left( 1-\frac{xh\left( n\right) }{y}\right) P_{n}^{g,h}\left( y\right) q^{n}
.
\]
Here, we extended the arithmetic
function $h$
by $h(0)=0$.
Since
$$\left( 1-x\sum _{n=1}^{\infty }g\left( n\right) q^{n}\right)
\sum _{n=0}^{\infty }Q_{n}^{g}\left( x\right) q^{n}=1,$$
we obtain
\begin{eqnarray*}
\sum _{n=0}^{\infty }P_{n}^{g,h}\left( y\right) q^{n}&=&\sum _{n=0}^{\infty }\left( 1-\frac{xh\left( n\right) }{y}\right) P_{n}^{g,h}\left( y\right) q^{n}
\sum _{m=0}^{\infty }Q_{m}^{g}\left( x\right) q^{m}\\
&=&\sum _{n=0}^{\infty }\sum _{k=0}^{n}\left( 1-\frac{xh\left( k\right) }{y}\right) P_{k}^{g,h}\left( y\right) Q_{n-k}^{g}\left( x\right) q^{n}
.
\end{eqnarray*}
Comparing coefficients of the last power series yields
\[
P_{n}^{g,h}\left( y\right) =\sum _{k=0}^{n}\left( 1-\frac{xh\left( k\right) }{y}\right) P_{k}^{g,h}\left( y\right) Q_{n-k}^{g}\left( x\right)
.
\]
Subtracting the terms in the last sum for $k=0$ and $k=n$ yields
(\ref{Vorzeichen})
\[
\frac{xh\left( n\right) }{y}P_{n}^{g,h}\left( y\right) -Q_{n}^{g}\left( x\right) =\sum _{k=1}^{n-1}\left( 1-\frac{xh\left( k\right) }{y}\right) P_{k}^{g,h}\left( y\right) Q_{n-k}^{g}\left( x\right)
.
\]

Let $n \geq 1$. We could truncate all involved sums to obtain (\ref{Vorzeichen}).
Therefore, the regularity of $\sum_{n=1}^{\infty} g(n) \, q^n$ at $q=0$ is not needed.
\end{proof}
\subsection{Proof of Theorem \ref{links} and Theorem \ref{rechts}}
\begin{proof}[Proof of Theorem \ref{links}.]
We prove
the
theorem by mathematical induction. The basic case $n=1$ is already proven.
Next, let $n>1$ and let the
theorem be true for all $1 \leq n_0  <n$.
Suppose $\kappa_{n_0} >0$
is given, such that for each $1 \leq m_0 \leq n_0$ and
$x < -\kappa_{n_0}$: $(-1)^{m_0} Q_{m_0}^g(x) >0$.
By induction hypothesis we have for $y < -\kappa_{n_0} \, H(n_0-1)$ and $1 \leq m_0 \leq n_0$:
\begin{equation*}
(-1)^{m_0} \, P_{m_0}^{g,h}(y)
>0.
\end{equation*}
By (\ref{ineq:uebertragung}) we obtain
\begin{eqnarray*}
&&\left( -1\right) ^{n
}\left( \frac{xh\left( n
\right) }{y}P_{n
}^{g,h}\left( y\right) -Q_{n
}^{g}\left( x\right) \right) \\
&=&\sum _{n_0=1}^{n-1}\left( 1-\frac{x h\left(
n_0\right) }{y}\right) \left( -1\right) ^{
n_0}P_{n_0
}^{g,h}\left( y\right) \left( -1\right) ^{
n-n_0
}Q_{
n-n_0
}^{g}\left( x\right)
\geq 0
\end{eqnarray*}
for $x
< -\kappa _{n}
\leq 0$ and $y
\leq xH\left(
n-1\right)
\leq xH
\left( n_0
\right)
$ for all $
n_0\leq n
-1$ as
$$H\left( n_0
\right) =\max \left\{ h\left(
n_0\right) ,H\left( n_0
-1\right) \right\} .$$
This implies
\begin{equation*}
(-1)^n \frac{x \, h(n)}{y} \, P_n^{g,h}(y)
\geq (-1)^n \, Q_n^g(x) >0,
\end{equation*}
and finally, the
theorem is proven.
\end{proof}

\begin{proof}[Proof of Theorem \ref{rechts}.]
Let $n \geq 2$ be given. Let $\frac{x \, h(k)}{y} >1$  for $1 \leq k \leq n-1$.
Then for each $1 \leq k \leq n-1$, the term  $
1 - \frac{x \, h(k)}{y}
$ is always negative. Since $h(n) \geq 1$,
this is guaranteed
since $\left\vert x \right\vert
> \left\vert y \right\vert $.
We prove the
theorem by mathematical induction. Theorem~\ref{rechts}
holds true for $n=1$ and $n=2$. Therefore, let $ n \geq 3$ and $\mu_n<0$ be given, such that for $ \mu_n < x <0$:
$Q_k^g(x) <0$ for all $1 \leq k \leq n$.
Then $P_k^{g,h}(x) < 0 $ for $1\leq k \leq n-1$ by induction hypothesis.
We examine (\ref{Vorzeichen}). Let $\mu_n <y <0$. Let $x$ be given,
such that $\mu_n < x <y <0$. Then we obtain from (\ref{Vorzeichen}) that
\begin{equation*}
\frac{x \, h(n)}{y} P_n^{g,h}(y) - Q_n^g(x) \leq 0.
\end{equation*}
Since $Q_{n}^{g}\left( x\right) <0$, the theorem is proven.
\end{proof}

\section{Final Remarks}
We begin with an example, which illustrates the refinement of previously
known results
on the zero domain of $P_{n}^{g,h}\left( x\right) $.
Then we state some identities deduced from the fundamental transfer lemma.
Finally,
we give a reformulation of Lehmer's conjecture and apply Theorem \ref{links} to Hermite polynomials $H_n(x)$.

\subsection{Example}
Let $g(n)=n$ if $n$ is odd and $g(n)=\frac{n}{2}$ if
$n$ is even. 
Let $h(n)=n$.
We deduce from \cite{HN21} that for $\kappa = 5.71$, we have
$$ P_n^{g,h} (z) \neq 0 \text{ for all } \vert z \vert > \kappa \, (n-1).$$
Numerical experiments indicate that 
$\kappa =4$ seems to be possible (Figure \ref{aa} and Table~\ref{betraege}).
We note that $Q_n^g(z)$ satisfies a $5$-term recursion.
From \cite{HNT20}, we obtain
\begin{equation*}
Q_n^g(z) = z \, Q_{n-1}^g(z) + (z+2) \, Q_{n+2}^g(z) + z \, Q_{n-3}^g(z) - Q_{n-4}^g(z), \qquad
n \geq
5
,
\end{equation*}
with initial values
$Q_1^g(z) = z$, $Q_2^g(z) = z^2 +z$,
$Q_3^g(z)= z^3 + 2 z^2 + 3 z$, and $Q_{4}^{g}\left( z\right) =z^4
 + 3\*z^3
 + 7\*z^2
 + 2\*z$.

\begin{figure}[H]
\includegraphics[width=.4\textwidth]{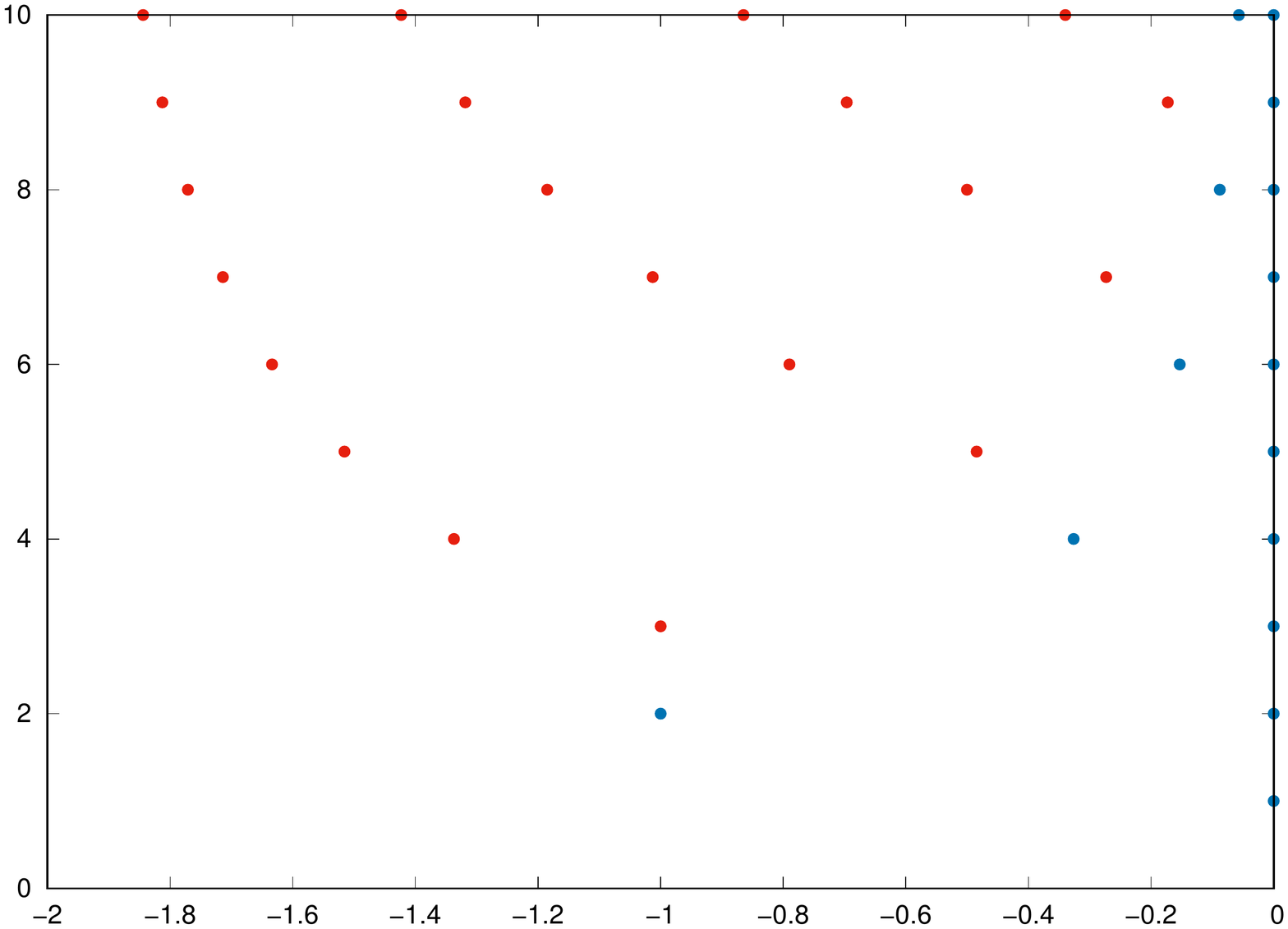}
\includegraphics[width=.4\textwidth]{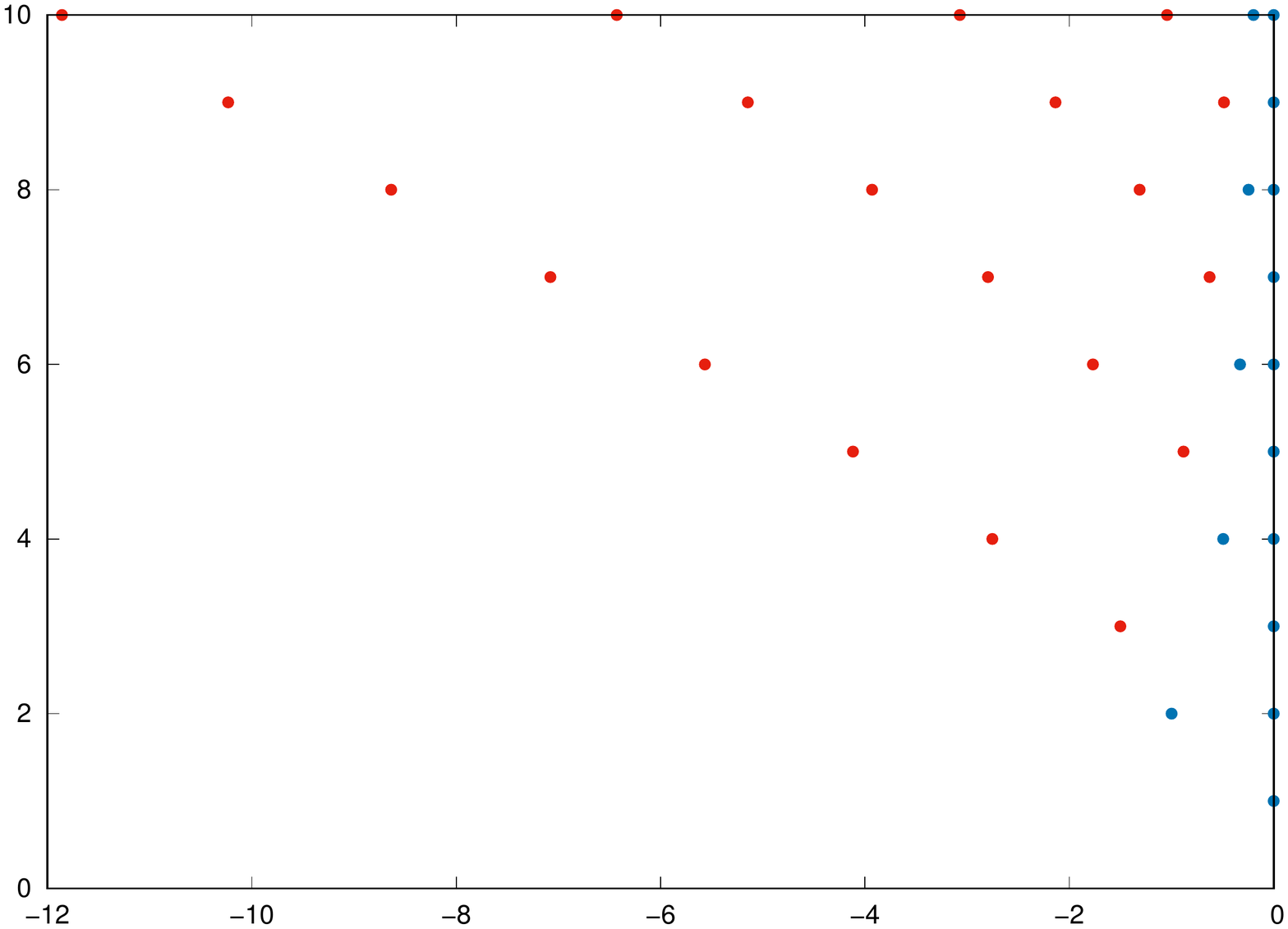}
\caption{\label{aa}Real parts of
zeros of $Q_{n}^{g} (x) $  and $P_n^g (x)$
for $1\leq n\leq 10$.
Here blue denotes that the zero is a real 
number and red denotes that it is not a real one.
}
\end{figure}

\begin{table}
\[
\begin{array}{rc|rc|}
\hline
n & \left| x\right| & n & \left| x\right| \\ \hline \hline
2 & 1.000000000 & 20 & 3.908878746 \\
3 & 1.732050808 & 30 & 3.958353612 \\
4 & 2.475342535 & 40 & 3.976255443 \\
5 & 2.910743051 & 50 & 3.984682343 \\
6 & 3.189361602 & 60 & 3.989307029 \\
7 & 3.374909553 & 70 & 3.992114849 \\
8 & 3.504161170 & 80 & 3.993946290 \\
9 & 3.597512290 & 90 & 3.995206626 \\
10 & 3.667001052 & 100 & 3.996110777 \\ \hline
\end{array}
\]
\caption{\label{betraege}Approximative largest absolut values of the zeros of $Q_{n}^{g}\left( x\right) $. }
\end{table}

In the following let $n \leq 100$. Then $
-1
$ is the smallest real zero.
Theorem \ref{links} implies that all real zeros of $P_n^{g,h}(z)$ are larger equal 
$ -(n-1)$. We also see that the largest real zeros of $\{Q_n^g(z)\}$ are monotonously increasing and tend to converge to $0$. Therefore, Theorem \ref{rechts} implies that for $\beta_n$
the largest non-trivial real zero of $\{Q_m^g(z)\}_{m \leq n}$,
the largest non-trivial real root $\tilde{\beta }_n$  of $\{P_n^g(z)\}_{m \leq n}$
satisfies $$\tilde{\beta_n} \leq \beta_n.$$

Note,
since $Q_n^g(z)/z$ and $n$ odd has no real zeros one may deduce in general that
this implies
$P_n^g(z)/z$
not to have real zeros as well. But this is
not the case in general.

{\bf Counterexample:}
Let $g(n) = \binom{\alpha \, n -1}{n-1}$ for $ \alpha \geq 1$. Then 
\begin{equation*}
P_n^g(z) = \frac{z}{n!} \prod_{k=1}^{n-1} \left( z + (\alpha -1) \, n +k \right).
\end{equation*} 
Let $\alpha= \frac{3}{2}$. Then the zeros of $Q_3^g(z)/z$ are not real, but $P_3^g(z)/z$ has obviously 
two real zeros.

\subsection{Identities}
Let $F_n$ be the Fibonacci numbers, where $F_0=0$, 
$F_1=1$, $F_2=1$, $F_3=2$, $F_4 = 3$.
It is known that $Q_n^{\func{id}}(1)= F_{2n}$ for $n \geq 1$ \cite{BHN22}. 
We obtain from the fundamental transfer lemma and special properties of $P_n^{\func{id}}(z)$ and $Q_n^{\func{id}}(z)$:
\begin{equation*}
L_{n-1}^{(1)}(y) - F_{2n} = - \sum_{k=1}^{n-1} \left( \frac{y}{k} +1 \right) L_{k-1}^{(1)}(y) \, F_{2(n-k)}.
\end{equation*}

\subsection{Lehmer's
Conjecture}
Let $\Delta \left( \omega
\right) := q \prod_{n=1}^{\infty} \left( 1 - q^n \right)^{24}$. Then the Ramanujan $\tau$-function is provided by the Fourier coefficients of $\Delta(\omega)$:
\begin{equation*}
\Delta(\omega) = \sum_{n=1}^{\infty} \tau(n) \, q^n.
\end{equation*}
Lehmer conjectured that $\tau(n)$
never vanishes. Let $n$ be the smallest natural number, such that $\tau(n) =0$. Then $n$ is a prime \cite{Le47}. Moreover, it is known that
Lehmer's conjecture is true for $n \leq 10^{23}$ \cite{DHZ14}.

There are several variations
of Lehmer's conjecture known, 
especially in the context of special values of certain polynomials 
\cite{On08}. In this spirit, we offer:
\newline
\newline
{\bf Lehmer's
Conjecture
(Polynomial Version)}\\
For all $n \geq 1$
and for every $z \in \mathbb{C}$:
\begin{equation}\label{HN Lehmer}
\sum_{k=0}^{n-1} \left( \frac{1}{z} + \frac{k}{24}\right) 
\, \tau(k+1) \, Q_{n-k}^{\sigma}(z) \neq 0.
\end{equation}
We have recorded in Table \ref{Q} the first values of $Q_{n}^{\sigma }\left(
-1\right) $, which makes it
easy to verify (\ref{HN Lehmer}) for  $z=-1$.
It would be interesting to have Lehmer's conjecture in terms of special values of $Q_n^{\sigma}(z)$, as $z=-24$ in the case of D'Arcais polynomials.

\begin{table}
\[
\begin{array}{r|rrrrrrrrrr}
\hline
n&1&2&3&4&5&6&7&8&9&10\\ \hline \hline
Q_{n}^{\sigma }\left( -1\right) &-1&-2&1&2&4&-6&-5&4&1&18  \\ \hline
n&11&12&13&14&15&16&17&18&19&20\\ \hline \hline
Q_{n}^{\sigma }\left( -1\right) &-13&-26&4&22&66&-76&-78&66&37&122\\ \hline
\end{array}
\]
\caption{ \label{Q}Values of $Q_{n}^{\sigma }\left( -1\right) $ for $1\leq n\leq 20$.}
\end{table}

\subsection{Bounds on the Zeros of Hermite Polynomials}
We denote by $H_n(x)$ the $n$th Hermite polynomials defined by (\cite{Sz75}, (5.5.4)):
\begin{equation*}
H_n(x):=n! \, \sum_{k=0}^{\floor{n/2
} }
\frac{(-1)^k}{k!} \, \frac{(2x)^{n-2k}}{(n-2k)!}.
\end{equation*}
We have $H_0(x)=1$, $H_1(x)=2x$, $H_2(x)=4x^2-2$, and $H_3(x)=8x^3-12x$.
The Hermite polynomials are orthogonal. The zeros are real, simple, and interlacing.
Theorem \ref{links} leads to the following result. We also refer to
(\cite{Sz75}, (6.326)).
\begin{corollary}
Let $n\geq 2$. Then for the zeros of $
H
_{n}\left( x\right) $
holds
\[
\left| x\right| \leq
\cos \left( \frac{\pi }{n+1}\right) \sqrt{
2n-2
}.
\]
\end{corollary}

\begin{proof}
Let $g\left( 1\right) =g\left( 2\right)=1 $ and $g\left( n\right) =0$ for
$n\geq 3$. Then
$$\sum _{n=0}^{\infty }Q_{n}^{g}\left( -x^{2}\right) q^{n}=\left( 1
-2\left( x/2\right) \left( -xq\right) +\left( -xq\right) ^{2}
\right) ^{-1}=\sum _{n=0}^{\infty }U_{n}\left( x/2\right) \left( -x\right) ^{n}q^{n}.$$
Therefore,
$Q_{n}^{g}\left( -x^{2}\right) =\left( -x\right) ^{n}U_{n}\left( x/2\right) $.
Note, that $U_{m}
\left( x/2\right) =0$ implies
$\left| x\right| \leq 2\cos \left( \frac{\pi }{n+1}\right) $
for all $m\leq n$.
This leads to
$
x
^{2}\leq 4\left( \cos \left( \frac{\pi }{n+1}\right) \right) ^{2}$.

We obtain
\[
\sum _{n=0}^{\infty }P_{n}^{g}\left( -2x^{2}\right) q^{n}=\exp \left( 2x\left( -x
q\right) -
\left( -xq
\right) ^{2}
\right) =\sum _{n=0}^{\infty }
\frac{H
_{n}\left( x\right) \left( -x
\right)
^{n}}{
n!}
q^{n}.
\]
Therefore,
$P_{n}^{g}\left( -2x^{2}\right) =
\left( -x\right) ^{n}
H_{n}\left( x\right) /\left( n!\right) $.
Theorem~\ref{links}
now ensures that
$$
2x
^{2}\leq 4\left( \cos \left( \frac{\pi }{n+1}\right) \right) ^{2}\left( n-1\right) .$$
Finally,
$\left| x\right| \leq
\cos \left( \frac{\pi }{n+1}\right) \sqrt{2n-2}$.
\end{proof}

\begin{remark}
It follows from (\cite{Sz75}, (6.326)) that the largest 
zero of $H_{n}\left( x\right) $ is bounded by
$\sqrt{2n+1}-6^{-1/2}
\left( 2n+1\right) ^{-1/6}
i_{1}$
with $i_{1}$ the smallest positive zero of Airy's function. 

In our case, we have
\begin{equation*}
\cos \left( \frac{\pi }{n+1}\right) \sqrt{2n-2}=
\sqrt{2n}-\left( 2n\right) ^{-1/2}+O
\left( n^{-3/2}\right) 
\end{equation*}
with Landau's $O$ notation.
The estimate we obtain from the transfer lemma
is
slightly weaker, since the second order term has exponent $-1/2$
instead of $-1/6$.
\end{remark}

{\bf Acknowledgments.}
We thank Professor Schmeisser for
his insightful comments.

\end{document}